\documentclass[a4paper,11pt]{article}

\usepackage{a4wide}
\usepackage[utf8]{inputenc}
\usepackage[T1]{fontenc}
\usepackage[english]{babel}
\usepackage{amsmath}
\usepackage{amssymb}
\usepackage{url}
\usepackage{multirow}
\usepackage{amsfonts}
\usepackage{amsthm}
\usepackage{indentfirst}
\usepackage{bbm}
\usepackage{color}
\usepackage{bigints}
\usepackage{graphicx}

\numberwithin{equation}{section}

\newtheorem{theorem}{Theorem}[section]

\newcommand{\ii}{{\mathrm{i}}}

\newcommand{\phib}{\boldsymbol\phib}

\newcommand{\I}{\mathbf{I}}
\newcommand{\A}{\mathbf{A}}

\def\C{\mathcal{C}}

\def\RR{\vbox {\hbox{I\hskip-2.1pt R\hfil}}}
\def\NN{\vbox {\hbox{I\hskip-2.1pt N\hfil}}}
\def\PP{\vbox {\hbox{I\hskip-2.1pt P\hfil}}}

\def\XXint#1#2#3{{\setbox0=\hbox{$#1{#2#3}{\int}$}
     \vcenter{\hbox{$#2#3$}}\kern-.5\wd0}}

\begin{document}
\title{ A product integration rule on equispaced nodes for highly oscillating integrals  \thanks{The authors are member of the INdAM Research group GNCS and the TAA-UMI Research Group. This research has been accomplished within ``Research ITalian network on Approximation'' (RITA).} }

\author{
Luisa Fermo\thanks{Department of Mathematics and Computer Science, University of Cagliari,  Via Ospedale 72, 09124, Italy (\tt fermo@unica.it)}
\and Domenico Mezzanotte \thanks{Department of Mathematics and Computer Science,
University of Basilicata, Viale dell'Ateneo Lucano 10, 85100
Potenza, Italy (\tt domenico.mezzanotte@unibas.it) }
\and Donatella Occorsio\thanks{Department of Mathematics and Computer Science, University of Basilicata, Viale dell'Ateneo Lucano 10, 85100 Potenza, Italy and
Istituto per le Applicazioni del Calcolo ``Mauro Picone'', Naples branch, C.N.R. National
Research Council of Italy, Via P. Castellino, 111, 80131 Napoli, Italy (\tt donatella.occorsio@unibas.it) }}

\maketitle

\begin{abstract}
This paper provides a product integration rule for highly oscillating integrands, based on equally spaced nodes. The stability and the error estimate are proven in the space of continuous functions, and some numerical tests which confirm such estimates are provided.

\medskip
\noindent{\bf Keywords}: Approximation by polynomials,
Boolean iterated sums of Bernstein operators,
quadrature rules
\medskip

\noindent{\bf Mathematics Subject Classification:} 41A10, 65D32
\end{abstract}

\section{Introduction}\label{sec:intro}
From a numerical point of view, the main difficulty to treat integrals of the type
\begin{equation}\label{T}
\int_{-a}^a e^{-\ii \omega (x-y)} f(x) dx, \quad a>0, \quad \ii=\sqrt{-1}, \quad y \in [-a,a]
\end{equation}
depends on the presence of the  kernel $e^{-\ii \omega(x-y)}$, which highly oscillates for  frequencies  $|\omega|>>1$. For the evaluation of such integrals, many accurate formulas exist in the literature, mainly based on the zeros of orthogonal polynomials; see e.g. \cite{Asheim,debonispastore,FermoMee2021,gautschi,HuybrechsVandewalle,Occorsio2018}.
On the other hand, in many practical applications, the function $f$ is known  only at  equispaced nodes.

In order to adopt the values of $f$ at equidistant nodes, a widely used technique is based on composite quadrature rules of ``lower" degree, such as  Fil\'on type rules,
 whose  degree of approximation cannot be improved over the saturation class of the approximation process.  Recently, in \cite{dellaccio} starting from the knowledge of $f$ at a finite set of equally spaced nodes,  an approach  has been introduced to efficiently compute
weighted integrals, but not advisable in the case of  oscillating integrands. For a short review on the main techniques the interested reader can consult \cite{Asheim}.

Finding accurate quadrature rules based on equidistant points is still an open problems. Here, we introduce a product integration rule based on the approximation of the function $f$ through   the generalized Bernstein (shortly GB) polynomials $\bar B_{m,\ell}(f)$ of degree $m$ and parameter $\ell$ defined in the interval $[-a,a]$.  Such polynomials represent an adequate tool for our aims, since they make use of the samples of $f$ at $m+1$ equally spaced nodes; see e.g. \cite{Felbecker,MO1977,Micchelli}. Moreover, unlike the classical Bernstein polynomials $\bar B_{m}(f)$,  the speed of the uniform convergence to $f$ accelerates as the smoothness of $f$ increases \cite{GonskaZhou}.

The outline of the paper is as follows. In Section \ref{section2} we provide definition and main properties of GB polynomials. Section \ref{section3} contains the main results regarding the product integration rules and its error estimates. In Section \ref{test} two numerical examples are given to corroborate the theoretical estimate. The proofs are given in the last Section \ref{proof}.

\section{Notation and Preliminary Results}\label{section2}
In the sequel $\C$ denotes any positive constant having different meanings at  different occurrences and the writing $\C\neq \C(a,b,..)$ has to be understood as $\C$ not depending on  $a,b,..$.

As usual,  $\mathbb{P}_m$ denotes the space of the algebraic polynomials of degree less than or equal to $m$, and
 $C^0:=C^0([-a,a])$ is the space of the continuous functions on $[-a,a]$  equipped with the uniform norm $\displaystyle \|f\|:=\max_{x\in [-a,a]}|f(x)|$. Morerover, for any bivariate function $g(x,y)$, by $g_x(y)$ we refer to $g$ as function of the only  variable $y$.
Finally, for a given integer $m$, we  set $N_0^m:= \{0,1,2,\dots,m\}.$

\subsection{Iterated Boolean sums of Bernstein Operators in $[-a,a]$}
For a fixed positive $a\in \RR$, consider the $m$-th Bernstein polynomial $\bar B_mf$  of a given  $f\in C^0,$
\begin{equation} \label{espressionebernstein}
\bar B_m(f,x):=\sum_{k=0}^m f(t_k) \bar p_{m,k}(x), \qquad  t_k:=-a+k\frac{2a}m, \qquad x \in [-a,a]
\end{equation}
where
\begin{equation} \label{def-pmk}
\bar p_{m,k}(x):=\binom m k \left(\frac{a+x}{2a}\right)^k
\left(\frac{a-x}{2a}\right)^{m-k}.
\end{equation}

Based on $\bar B_m (f)$, the  $\ell$-iterated boolean sum $\bar B_{m,\ell}(f)\in \PP_m$, $\ell\in \NN$,  is defined for any $f\in C^0$ as
 $$\bar B_{m,\ell}(f)=f-(f-\bar B_m(f))^\ell,  \quad \bar B_{m,1}f=\bar B_{m}f .$$

The polynomial $\bar B_{m,\ell}(f)$ can be expressed $\forall x\in [-a,a]$, as
\begin{equation}\label{bmlapoly}
\bar B_{m,\ell}(f,x)=\sum_{j=0}^m \bar p_{m,j}^{(\ell)}(x)f(t_j), \quad \textrm{with} \quad \bar p_{m,j}^{(\ell)}(x)=\sum_{i=0}^m \bar{p}_{m,i}(x)c_{i,j}^{(m,\ell)},
\end{equation}
where $c_{i,j}^{(m,\ell)}$ are the entries of the matrix $C_{m,\ell}\in \RR^{(m+1)\times (m+1)}$,
\begin{eqnarray*}
C_{m,\ell} =\I+(\I-\A)+\ldots+(\I-\A)^{\ell-1},\quad C_{m,1}=\I,
\end{eqnarray*}
being  $\I$ the identity matrix of order $m+1$ and $\A\in \mathbb{R}^{(m+1)\times(m+1)}$ the matrix
\[(\A)_{i,j}=p_{m,j}(t_i),\quad  i,j \in N_0^m .
\]

By induction  on $\ell=2^p,\ p\in \NN$, the following recurrence relation holds
\begin{equation}\label{matrice_potenzedi2}
C_{m,2^p}=C_{m,2^{p-1}}+(\I-\A)^{2^{p-1}}C_{m,2^{p-1}},
\end{equation}
which allows to a fast construction of the subsequence $\{B_{m,2^p}\}_{p=1,2,\ldots}$, since
\begin{equation}\label{relazioneutile}
B_{m,2^p}(f,x)=2B_{m,2^{p-1}}(f,x)-  B_{m,2^{p-1}}^2(f,x).
\end{equation}

A  survey containing properties and applications of Generalized Bernstein polynomials is given in \cite{OccoRussoThem}. Moreover, specific employments of such polynomials to Fredholm and Volterra integral equations can be found in \cite{FMO2022,OccorsioRusso2014}.

As we can note by \eqref{espressionebernstein}, the polynomials $\bar{B}_{m,\ell} f$ require the same samples of $f$ at  $m+1$ equally spaced  points of $[-a,a]$. Moreover, they detain the special property that the speed of convergence to $f$ accelerates as the smoothness of $f$ increases, unlike the ``originating''  polynomial $\bar B_{m}(f)$.

To be more precise, if we consider the $r$th Sobolev type space  $1 \leq r\in\NN$,   defined as
\[
W_r=\left\{f\in C^0: f^{(r-1)}\in \mathcal{AC} : \|f^{(r)}\phi^r\|<\infty\right\},\quad\ \|f\|_{W_r}=\|f\|+\|f^{(r)}\phi^r\|,\] where $\phi(x)=\sqrt{a^2-x^2},$ and $\mathcal{AC}$ denotes the space of all locally absolutely continuous functions on $[-a, a]$,  then for each $f\in  W_r$ with $0<r \leq 2\ell$, we have \cite{GonskaZhou}
$$\|f-\bar B_{m,\ell}(f)\| = \mathcal{O}(\sqrt{m^{-r}}).$$ In other words, the rate of convergence behaves like the square root of the  best approximation error for this space of functions.

\section{A product integration rule}\label{section3}
By separating the kernel $e^{-\ii \omega (x-y)}$ into the real and imaginary parts, we are dealing with integrals of the type
\begin{equation}\label{inte_osc}
\mathcal{I}^\omega(f,y):=\int_{-a}^a \kappa(\omega(y-x)) \ f(x) \,  dx, \qquad \textrm{where} \qquad  \kappa(\omega(y-x))=\begin{cases}\sin(\omega(y-x)),\\ \cos(\omega(y-x)).\end{cases}\end{equation}
In order to evaluate numerically the above integral, we propose to approximate  $f$ by $\bar B_{m,\ell}(f)$, getting
\begin{equation}
\mathcal{I}^\omega( f,y)=\int_{-a}^a  k(\omega(y-x)) \, \bar B_{m,\ell}(f,x) \, dx+ \mathcal{R}_{m,\ell}^\omega (f,y):=\mathcal{I}_{m,\ell}^\omega (f,y)+ \mathcal{R}_{m,\ell}^\omega (f,y).
\end{equation}
Then, taking \eqref{bmlapoly} into account, we get
\begin{equation}\label{vm}
\mathcal{I}_{m,\ell}^\omega (f,y)=\sum_{j=0}^m f(t_j)\sum_{i=0}^m c_{i,j}^{(m,\ell)} \int_{-a}^a  \kappa(\omega(y-x)) \bar p_{m,i}(x) dx=:\sum_{j=0}^m f(t_j)\sum_{i=0}^m c_{i,j}^{(m,\ell)}q_i(y).
\end{equation}
Let us note that, in the above formula, the patology of the integrand, that is the presence of the oscillating kernel, has been isolated in the coefficients $q_i$ that we approximate by using a suitable tecnique explained in the next subsection.

Next theorem establishes conditions assuring that the rule \eqref{vm} is stable and convergent for any $f\in C^0$. Moreover, it also provides an error estimate in suitable subspaces of
$C^0$.
\begin{theorem}\label{theo0}
For any $f\in C^0$, and for any $k$ defined as in \eqref{inte_osc}, and for each fixed $\ell$, the rule \eqref{vm} is stable, that is
\begin{equation}\label{stab_rule}
\sup_m \|\mathcal{I}_{m,\ell}^\omega (f)\|\le \C \|f\|,\quad \C\neq \C(m,f).
\end{equation}
Moreover, for any $f\in W_r$, $0<r\le 2\ell$, we have
\begin{equation}\label{conv_rule}
\| \mathcal{R}_{m,\ell}^\omega (f)\|\le  \C \left(\frac{a^{r+1} \|f\|_{W_r}}{(\sqrt{m})^r}\right), \quad \C\neq \C(m,f).
\end{equation}
\end{theorem}

\subsection{The approximation  of the coefficients}
Now we approach to the computation of the coefficients $q_i(y)$, by proposing a suitable technique that treats the pathology of the kernel.

Let $N=\left\lfloor\omega\frac{a}{\pi}\right\rfloor+1$, introduce the partition  $$[-a,\,a]=\bigcup_{h=1}^N [x_{h-1}, x_h], \quad x_0=-a, \quad x_h= -a+h\eta, \quad \eta=\frac{2a}{N}, \quad h=1,2,\dots,N,$$
and consider the following decomposition in sum
$$
q_i(y)=\sum_{h=1}^N \int_{x_{h-1}}^{x_h} \kappa(\omega(y-x)) \, \bar p_{m,i}(x)dx, \quad i=1,...,m.$$
Let us now map each interval $[x_{h-1},\,x_h]$ into $[-1,\,1]$ through the  linear transformations $ z=\gamma_h(x):= 2\frac{x-x_{h-1}}{\eta}-1.$
In this way we get
$$
q_i(y)=\frac{\eta}2\sum_{h=1}^N \int_{-1}^1 \bar p_{m,i}\left(\gamma_h^{-1}(z)\right) \kappa\left(\omega\left(y-\gamma_h^{-1}(z)\right)\right) dz.$$
Now, by approximating  each integral by a $m$-point Gauss-Legendre, we have
\begin{equation}\label{coeff_final}
q_{m,i}(y) = \frac{a}{N}\sum_{h=1}^N \left(\sum_{k=1}^m  \lambda_k \, \bar p_{m,i}\left(\gamma_h^{-1}(z_k)\right) \kappa\left(\omega\left(y-\gamma_h^{-1}(z_k)\right)\right)+\varepsilon_{m}^{i,h}(y)\right).
\end{equation}
where $\{z_k\}_{k=1}^{m}$ are the zeros of the $m$th Legendre polynomial  and  $\{\lambda_k\}_{k=1}^{m}$ are the Christoffel numbers.
\begin{theorem}\label{theo1}
Fixed $y\in [-a,a]$, for any  $1\le h\le N$ and $0\le i\le m$, we get
$$|\varepsilon_{m}^{i,h}(y)|\le \frac{\C}{\sqrt{2\pi i}}\left(\frac{a}N\right)^{m} \left(m^{\frac 1 {2m}}\cdot\frac{m+\omega}{2m-1}\right)^{m},$$
$\C\neq \C(m,f)$.
\end{theorem}

Hence, by replacing the approximation \eqref{coeff_final} into \eqref{vm}, we have the following product integration rule
\begin{equation}\label{formula_prodotto}
\widetilde{\mathcal{I}}^\omega_{m,\ell}(f,y)=\frac{a}{N}\sum_{j=0}^m f(t_j)\sum_{i=0}^m c_{i,j}^{(m,\ell)} \left( \sum_{h=1}^N \sum_{k=1}^m  \lambda_k \bar p_{m,i}  \left(\gamma_h^{-1}(z_k)\right) \kappa\left(\omega\left(\gamma_h^{-1}(y-z_k)\right)\right) \right).
\end{equation}

Combining estimates by Ths. \ref{theo0}\,-\,\ref{theo1}, assuming $f\in  W_r$, with $1\le r \le 2\ell$ ,  for any $y\in [-a,a]$ and for $m$ sufficiently large (say $m>m_0$, $m_0$ fixed), the following error estimate holds:
$$|\mathcal{I}^\omega(f,y)-\widetilde{\mathcal{I}}^\omega_{m,\ell}(f,y)|\le \C \|f\|_{W_r} \left[\frac{a^{r+1} }{(\sqrt{m})^r}+m^{\frac3 2} \|C_{m,\ell}\|_\infty \left(\frac{a}N \cdot\frac{m+\omega}{2m-1}\right)^{m}\,\right], $$
where $\C\neq \C(m,f).$

\section{Numerical tests}\label{test}
In this section, we show the accuracy of our product rule as well as the reliability of our theoretical estimates, to the following two integrals
$$\mathcal{I}^\omega(f_1,y)=\int_{-1}^1  \sin{\omega(y-x)} \, f_1(x) \, dx, \qquad \mathcal{I}^\omega(f_2,y)=\int_{-2}^2  \cos{\omega(y-x)} \, f_2(x) dx,$$
with $f_1(x)=\tanh{(x+1)}$, and $f_2(x)=|x+1|^{9/2}.$
In both cases, the exact value of the integrals is not known and then we consider as exact the value provided by our quadrature rule with $m=512$. Table \ref{table}  contains, for increasing values of $m$,  the absolute errors
$$e^\omega_{m,\ell}(f_j, y)= |\widetilde{\mathcal{I}}_{512,\ell}^\omega(f_j,y)-\widetilde{\mathcal{I}}^\omega_{m,\ell}(f_j,y)|, \qquad j=1,2$$
for the first and second integral, respectively. In both text, we fix $\ell=2^8$ and we show the results in two different points and for several values of $\omega$.

\begin{table}[h]
	\caption{Numerical results for integral $\mathcal{I}^\omega(f_1,y)$ to the left and for integral $\mathcal{I}^\omega(f_2,y)$ to the right \label{table}}
	\centering
	\begin{tabular}{c|ccc}
		$m$ & $\omega$ & $e^\omega_{m,2^8}(f_1, -0.7)$ & $e^\omega_{m,2^8}(f_1, 0.5)$ \\ \hline
	4 & 10	& 7.20e-04 &	 1.30e-04 \\
8 &	& 1.63e-05 &	 2.25e-05\\
16 &	& 5.91e-09 &	 4.15e-09\\
32 &	& 2.82e-12 	& 8.81e-12\\
64 &	& 1.32e-15 &	 7.42e-15\\
128 & & 8.26e-15 &	 3.23e-16\\
256 & & 3.27e-15 &	 7.28e-15\\
\hline
4 & $10^2$ &	 8.71e-04 &	 1.17e-04 \\
8 &	&	 2.38e-07 &	 1.06e-07 \\
16 &	&	 1.37e-10 &	 1.17e-10 \\
32 &	&	 2.47e-12 &	 4.09e-14 \\
64 &	&	 2.99e-15 &	 7.88e-16 \\
128 &	& 3.07e-15 &	 8.41e-16 \\
256 &	& 1.99e-15 &	 1.90e-15 \\
\hline
4 & $10^3$ &	 6.93e-04 &	 6.64e-04\\
8 &   &	 1.53e-09 &	 1.44e-09\\
16 &   &	 3.97e-12 &	 4.25e-12\\
32 	&   & 4.63e-14 &	 5.18e-14\\
64 	&   & 2.45e-15 &	 1.53e-16\\
128 &   & 1.15e-15 &	 1.07e-15\\
256 &   & 7.71e-16 &	 9.89e-16\\
\hline
	\end{tabular}
\qquad \begin{tabular}{c|ccc}
		$m$ & $\omega$ & $e^\omega_{m,2^8}(f_2, -1.5)$ & $e^\omega_{m,2^8}(f_2, 1)$ \\ \hline
	4 	& 10 &     1.98e-02 &	 6.34e-03 \\
8 	  &  & 3.78e-03 &	 3.47e-03\\
16 	  &  & 1.88e-04 &	 1.79e-04\\
32 	  &  & 2.41e-05 &	 2.75e-05\\
64 	  &  & 3.64e-08 &	 1.30e-07\\
128   &  & 7.76e-10 &	 3.92e-10\\
256   &  & 1.01e-11 &	 4.75e-12\\
\hline
4 	 & $10^2$ &    1.36e-02 &	 7.04e-02 \\
8 	 &   & 4.20e-05 	& 2.82e-06 \\
16 	 &   & 7.27e-06 	& 1.52e-06 \\
32 	 &   & 1.08e-07 	& 1.60e-07 \\
64 	 &   & 2.05e-07 	& 1.66e-07 \\
128 &	 & 4.32e-09 	& 4.58e-09 \\
256 &	 & 1.36e-10 	& 1.07e-10 \\
\hline
4 	& $10^3$ &    6.91e-02 &	 2.67e-02 \\
8 	&   &  3.79e-07 &	 4.19e-07 \\
16 	&   &  6.53e-08 &	 7.24e-08 \\
32 	&   &  1.77e-09 &	 1.90e-09 \\
64 	&   &  2.08e-09 &	 2.44e-09 \\
128 &	&  2.63e-11 &	 2.94e-11 \\
256 &	&  2.06e-12 &	 2.67e-12 \\
\hline
	\end{tabular}
\end{table}

As we can observe, in the first integral the convergence is very fast due to the analycity of the function $f_1$. On the contrary, in the second integral the convergence is slower since $f_2 \in W_{9/2}$.
\section{The proofs}\label{proof}
\begin{proof}[Proof of Theorem \ref{theo0}]
First, let us prove the stability, i.e. \eqref{stab_rule}. By the first identity of \eqref{vm}, for any  $y\in [-a,a]$ we have
\begin{align*}
|\mathcal{I}_{m,\ell}^\omega (f,y)|&\le  \|f\|\int_{-a}^a \sum_{j=0}^m \left|\sum_{i=0}^m c_{i,j}^{(m,\ell)} \bar p_{m,i}(x) \right| dx \le
\|f\|\int_{-a}^a \sum_{i=0}^m \bar p_{m,i}(x) \sum_{j=0}^m  \left|c_{i,j}^{(m,\ell)}\right|  dx \\ & \le 2a \,  \|f\|\ \, \|C_{m,\ell}\|_\infty \le \C \, \|f\|,
\end{align*}
since $\|C_{m,\ell}\|_\infty\leq 2^\ell-1.$

Let us now estimate the error. We have
$$|\mathcal{R}_{m,\ell}^\omega (f,y)|\le \int_{-a}^a |f(x)-\bar B_{m,\ell}(f,x)| \, k(\omega(y-x))dx \le \C \, \|f-\bar B_{m,\ell}(f)\|.$$
The thesis follows taking into account \cite[Th. 2.1]{GonskaZhou} .
\end{proof}

\begin{proof}[Proof of Theorem \ref{theo1}]
First, let us recall that that for any $g\in W_{m}([-1,1])$  the error of the Gauss-Legendre rule is given by
$$\left| \int_{-1}^1 g(x) dx-\sum_{k=1}^m \lambda_k g(y_k) \right| \le \C E_{2m-1}(f)\le \frac{\C}{(2m-1)^m} ( \|g\|+\|g^{(m)}\varphi^{m}\|),\quad \varphi(x)=\sqrt{1-x^2}$$
where here and in the next $\|g\|= \displaystyle \sup_{z\in [-1,1]}|g(z)|$.
Hence, we can write
\begin{equation*}|
\varepsilon_{m}^{i,h}(y)| \le  \frac{\C}{(2m-1)^{m}}\left\{ \left\|\bar p_{m,i}\left(\gamma_h^{-1}\right) \kappa_y\left(\omega\left(\gamma_h^{-1}\right)\right)\right\|+\left\|\left[ \bar p_{m,i}(\gamma_h^{-1}) \kappa_y(\omega(\gamma_h^{-1}))\right]^{(m)} \varphi^m \right\|\right\}.
\end{equation*}
Now, taking into account that for $1\le i\le m-1$ 
$$\|\bar p_{m,i}\|=\binom m i \left(\frac i m\right)^i \left(\frac{m-i}m\right)^{m-i}$$
by Stirling's formula we have $$\|\bar p_{m,i}\|\sim \frac {\sqrt{m}}{\sqrt{2\pi i}} \quad \textrm{as} \quad  m\to \infty,$$
and by  Bernstein inequality for polynomials (see e.g.\cite{mastrobook}) we deduce for $0\le r\le m, \ r\in \NN$
$$\left\|\bar p_{m,i}^{(m-r)}(\gamma_h^{-1})\varphi^{m-r}\right\|\le \C\left(\frac{a}N\right)^{m-r} m^{m-r}\|\bar p_{m,i}\|\le \frac{\C}{\sqrt{2\pi i}}  {\sqrt{m}} \left(\frac{a}N\right)^{m-r} m^{m-r}.$$
Moreover, by the definition of the kernel $k_y$ in \eqref{inte_osc} we have
$$ \left\|\kappa_y^{(r)}\left(\omega\left(\gamma_h^{-1}\right)\right)\right\|\le \left( \omega \frac {a}N\right)^r.
$$
Consequently, by applying  Leibniz formula, we get
\begin{align*}
\left\| \left(\bar p_{m,i}\left(\gamma_h^{-1}\right) \kappa_y\left(\omega\left(\gamma_h^{-1}\right)\right)\right)^{(m)} \right\|&\le  \C\frac{\sqrt{m}}{\sqrt{2\pi i}} \left(\frac{a}N\right)^{m}\sum_{r=0}^{m}\binom{m}{r} m^{m-r}\ \omega^r \\ & = \C\frac{\sqrt{m}}{\sqrt{2\pi i}} \left(\frac{a}N\right)^{m}(m+\omega)^{m},
\end{align*}
and then
\begin{align*}
|\varepsilon_{m}^{i,h}(y)| &\le \C\frac{\sqrt{m}}{\sqrt{2\pi i}}\left(\frac{a}N\right)^{m} \left(\frac{m+\omega}{2m-1}\right)^{m}.
\end{align*}
\end{proof}

\bibliographystyle{plain}      
\bibliography{biblio}

\end{document}